\documentclass{article}
\usepackage{amsmath}
\usepackage{amsfonts}
\usepackage{amsthm}
\usepackage{amssymb}
\usepackage{color}

\newtheorem{theorem}{Theorem}[section]

\newtheorem{observation}[theorem]{Observation}

\newtheorem{corollary}[theorem]{Corollary}

\theoremstyle{definition}
\newtheorem{definition}[theorem]{Definition}
\theoremstyle{remark}
\newtheorem{remark}[theorem]{Remark}

\def\f2{\mathbb{F}_2}

\newcommand{\RNP}{{\rm RNP}}

\newcommand{\ep}{\varepsilon}

\begin{document}

\title{Radon-Nikod\'ym property and thick families of geodesics}

\author{Mikhail Ostrovskii}

\date{\today}
\maketitle

\begin{large}

\noindent{\bf Abstract.} Banach spaces without the Radon-Nikod\'ym
property are characterized as spaces containing bilipschitz images
of thick families of geodesics defined as follows. A family $T$ of
geodesics joining points $u$ and $v$ in a metric space is called
{\it thick} if there is $\alpha>0$ such that for every $g\in T$
and for any finite collection of points $r_1,\dots,r_n$ in the
image of $g$, there is another $uv$-geodesic $\widetilde g\in T$
satisfying the conditions: $\widetilde g$ also passes through
$r_1,\dots,r_n$, and, possibly, has some more common points with
$g$. On the other hand, there is a finite collection of common
points of $g$ and $\widetilde g$ which contains $r_1,\dots,r_n$
and is such that the sum of maximal deviations of the geodesics
between these common points is at least $\alpha$.

\section{Introduction}

The Radon-Nikod\'ym property (\RNP) is one of the most important
isomorphic invariants of Banach spaces. We refer to \cite{BL00,
Bou79, Bou83, DU77, Pis11} for systematic presentations of results
on the \RNP, and \cite{DM13}, \cite{DM07} for recent work on the
\RNP.\medskip

In the recent work on metric embeddings a substantial role is
played by existence and non-existence of bilipschitz embeddings of
metric spaces into Banach spaces with the \RNP, see
\cite{CK06,CK09,LN06}. At the seminar ``Nonlinear geometry of
Banach spaces'' (Texas A \&\ M University, August 2009) Bill
Johnson suggested the problem of metric characterization of
reflexivity and the Radon-Nikod\'{y}m property \cite[Problem
1.1]{Tex09}. Some work on this problem was done in \cite{Ost11}
and \cite{Ost13+}. The purpose of this paper is to continue this
work. More precisely, we are going to characterize the \RNP\ using
thick families of geodesics defined in the following way.

\begin{definition}\label{D:ThickFam} Let $u$ and $v$ be two elements in a metric space $(M,d_M)$. A {\it $uv$-geodesic} is a
distance-preserving map $g:[0,d_M(u,v)]\to M$ such that $g(0)=u$
and $g(d_M(u,v))=v$ (where $[0,d_M(u,v)]$ is an interval of the
real line with the distance inherited from $\mathbb{R}$). A family
$T$ of $uv$-geodesics is called {\it thick} if there is $\alpha>0$
such that for every $g\in T$ and for any finite collection of
points $r_1,\dots,r_n$ in the image of $g$, there is another
$uv$-geodesic $\widetilde g\in T$ satisfying the conditions:

\begin{itemize}

\item The image of $\widetilde g$ also contains $r_1,\dots,r_n$.

Therefore there are $t_1,\dots,t_n\in [0,d_M(u,v)]$ such that
$r_i=g(t_i)=\widetilde g(t_i)$.

\item There are two sequences $\{q_i\}_{i=1}^m$ and
$\{s_i\}_{i=1}^{m+1}$ in $[0,d_M(u,v)]$ which are listed in
non-decreasing order and satisfy the conditions:

\begin{enumerate}

\item $\{q_i\}_{i=1}^m$ contains $\{t_i\}_{i=1}^n$

\item Points $s_1,\dots, s_{m+1}$ satisfy
\[0\le s_1\le q_1\le s_2\le q_2\le\dots\le s_m\le q_m\le
s_{m+1}\le d_M(u,v).\]

\item $g(q_i)=\widetilde g(q_i)$ for all $i=1,\dots, m$,  and

\[\sum_{i=1}^{m+1}d_M(g(s_i),\widetilde g(s_i))\ge\alpha.\]

\end{enumerate}

\end{itemize}

\end{definition}

The main purpose of this paper is to prove the following result.

\begin{theorem}\label{T:CustThick} For each non-\RNP\ Banach space $X$ there exists a metric space $M_X$ containing a thick family $T_X$
of geodesics which admits a bilipschitz embedding into $X$.
\end{theorem}

This result complements the following two results of
\cite{Ost13+}:

\begin{theorem}[\cite{Ost13+}]\label{T:ThickNRNP} If a Banach space $X$ admits a
bilipschitz embedding of a thick family of geodesics, then $X$
does not have the \RNP.
\end{theorem}

\begin{theorem}[\cite{Ost13+}]\label{T:UsingBR} For each thick family $T$ of geodesics there exists a Banach space $X$ which does not have the \RNP\ and does not
admit a bilipschitz embedding of $T$ into $X$.
\end{theorem}

\begin{corollary}[of Theorems \ref{T:CustThick} and
\ref{T:ThickNRNP}] A Banach space $X$ does not have the \RNP\ if
and only if it admits a bilipschitz embedding of some thick family
of geodesics.
\end{corollary}

\begin{remark} Theorem \ref{T:UsingBR} shows that the metric space $M_X$ and
the family $T_X$ in Theorem \ref{T:CustThick} should depend on the
space $X$.
\end{remark}

We use some standard definitions of the Banach space theory and
the theory of metric embeddings, see \cite{Ost13}.

\section{Proof of the main result}

\begin{proof}[Proof of Theorem \ref{T:CustThick}] We use the characterization of the \RNP\ in terms of bushes (see \cite[bottom of page~111]{BL00} or
\cite[Theorem 2.3.6]{Bou83}).

\begin{definition}\label{D:Bush} Let $Z$ be a Banach space and let $\ep>0$. A
set of vectors $\{z_{n,j}\}_{n=0,j=1}^{~\infty~~m_n}$ in $Z$ is
called an $\ep$-{\it bush} if for every $n\ge 1$ there is a
partition $\{A^n_k\}_{k=1}^{m_{n-1}}$ of $\{1,\dots,m_n\}$ such
that
\begin{equation}
||z_{n,j}-z_{n-1,k}||\ge \ep
\end{equation}
for every $j\in A^n_k$, and \begin{equation}\label{E:Bush2}
z_{n-1,k}=\sum_{j\in A^n_k}\lambda_{n,j}z_{n,j}\end{equation} for
some $\lambda_{n,j}\ge 0$, $\sum_{j\in A^n_k}\lambda_{n,j}=1$.
\end{definition}

The mentioned characterization of the \RNP\ is:

\begin{theorem}[\cite{BL00,Bou83}]\label{T:RNPBush} A Banach space $Z$ does not have the \RNP\ if and
only if it contains a bounded $\ep$-bush for some $\ep>0$.
\end{theorem}

In this theorem and below we may and shall assume that $m_0=1$.
\medskip

It is easy to see that the direct sum of two Banach spaces with
the \RNP\ has the \RNP. Because of this a subspace of codimension
$1$ in a non-\RNP\ Banach space also does not have the \RNP. Let
$x^*\in X^*$, $||x^*||=1$ be a functional which attains its norm
on $x\in X$, $||x||=1$. By Theorem \ref{T:RNPBush}, we can find a
bounded $\ep$-bush in $\ker x^*$. Shifting this bush by $x$ we get
a bush $\{x_{n,j}\}_{n=0,j=1}^{~\infty~~m_n}$ satisfying the
condition $x^*(x_{n,j})=1$ for all $n$ and $j$. Consider the
closure of the convex hull of the set $B_X\cup\{\pm
x_{n,j}\}_{n=0,j=1}^{~\infty~~m_n}$, where $B_X$ is the closed
unit ball of $X$. It is clear that this set is the unit ball of
$X$ in an equivalent norm and that in this new norm
\begin{equation}\label{E:Norm1}||x_{n,j}||=1\hbox{ for all }
n\hbox{ and }j.\end{equation} Since the property of $X$ which we
are going to establish is clearly an isomorphic invariant, it
suffices to consider the case where \eqref{E:Norm1} is satisfied.
\medskip

We are going to use this $\ep$-bush to construct a thick family of
geodesics in $X$ joining $0$ and $x_{0,1}$. First we construct a
subset of the desired set of geodesics, this subset will be
constructed as the set of limits of certain broken lines in $X$
joining $0$ and $x_{0,1}$. The constructed broken lines are also
geodesics (but they do not necessarily belong to the family
$T_X$).
\medskip

The mentioned above broken lines will be constructed using
representations of the form $x_{0,1}=\sum_{i=1}^mz_i$, where $z_i$
are such that $||x_{0,1}||=\sum_{i=1}^m||z_i||$. The broken line
represented by such finite sequence $z_1,\dots,z_m$ is obtained by
letting $z_0=0$ and joining $\sum_{i=0}^kz_i$  with
$\sum_{i=0}^{k+1}z_i$ with a line segment for $k=0,1,\dots,m-1$.
Vectors $\sum_{i=0}^k z_i$, $k=0,1,\dots,m$ will be called {\it
vertices} of the broken line.
\medskip

The infinite set of broken lines which we construct is labelled by
vertices of the infinite binary tree $B$ in which each vertex is
represented by a finite (possibly empty) sequence of $0$ and $1$,
two vertices in $B$ are adjacent if the sequence corresponding to
one of them is obtained from the sequence corresponding to the
other by adding one term on the right. (For example, vertices
corresponding to $(1,1,1,0)$ and $(1,1,1,0,1)$ are
adjacent.)\medskip

The broken line corresponding to the empty sequence $\emptyset$ is
represented by the one-element sequence $x_{0,1}$, so it is just a
line segment joining $0$ and $x_{0,1}$.
\medskip

We have
\[x_{0,1}=\lambda_{1,1}x_{1,1}+\dots+\lambda_{1,m_1}x_{1,m_1},\]
where $||x_{1,j}-x_{0,1}||\ge \ep$ (recall that we assumed
$m_0=1$). We introduce the vectors
\[y_{1,j}=\frac12(x_{1,j}+x_{0,1}).\]

For these vectors we have
\[x_{0,1}=\lambda_{1,1}y_{1,1}+\dots+\lambda_{1,m_1}y_{1,m_1},\]
$||y_{1,j}-x_{1,j}||=||y_{1,j}-x_{0,1}||\ge\frac{\ep}2$, and
$||y_{1,j}||=1$.\medskip

As a preliminary step to the construction of the broken lines
corresponding to one-element sequences $(0)$ and $(1)$ we form a
broken line represented by the points
\begin{equation}\label{E:1stSeq}\lambda_{1,1}y_{1,1},\dots,\lambda_{1,m_1}y_{1,m_1}.\end{equation}
We label the broken line represented by \eqref{E:1stSeq} by
$\overline{\emptyset}$.\medskip

The broken line corresponding to the one-element sequence $(0)$ is
represented by the sequence obtained from \eqref{E:1stSeq} if we
replace each term $\lambda_{1,j}y_{1,j}$ by a two-element sequence
\begin{equation}\label{E:(0)}\frac{\lambda_{1,j}}2x_{0,1},\frac{\lambda_{1,j}}2x_{1,j}.\end{equation}

The broken line corresponding to the one-element sequence $(1)$ is
represented by the sequence obtained from \eqref{E:1stSeq} if we
replace each term $\lambda_{1,j}y_{1,j}$ by a two-element sequence
\begin{equation}\label{E:(1)}\frac{\lambda_{1,j}}2x_{1,j}, \frac{\lambda_{1,j}}2x_{0,1}.\end{equation}

It is easy to see that if we let $g$ and $\widetilde g$ be the
broken lines corresponding to $(0)$ and $(1)$, respectively,
parameterized by $[0,1]$ using either the distance to $0$ or
values of $x^*$; and pick $s_1,\dots,s_{m+1}$ in such a way that
they correspond to ends of line segments determined by
$\frac{\lambda_{1,j}\ep}2x_{0,1}$ for $(0)$ and
$\frac{\lambda_{1,j}\ep}2x_{1,j}$ for $(1)$ (see \eqref{E:(0)} and
\eqref{E:(1)}), we get
\begin{equation}\label{E:SumDeviat}\sum_{i=1}^{m+1}||g(s_i)-\widetilde
g(s_i)||\ge\frac{\ep}2.\end{equation}

Broken lines corresponding to $2$-element sequences are also
formed in two steps. To get the broken lines labelled by $(0,0)$
and $(0,1)$ we apply the described procedure to the geodesic
labelled $(0)$, to get the broken lines labelled by $(1,0)$ and
$(1,1)$ we apply the described procedure to the geodesic labelled
$(1)$.\medskip

In the preliminary step we replace each term of the form
$\frac{\lambda_{1,k}}2x_{0,1}$ by a multiplied by
$\frac{\lambda_{1,k}}2$ sequence \eqref{E:1stSeq}. We replace a
term of the form  $\frac{\lambda_{1,k}}2x_{1,k}$ by the multiplied
by $\frac{\lambda_{1,k}}2$ sequence
\begin{equation}
\{\lambda_{2,j}y_{2,j}\}_{j\in A^2_k},
\end{equation}
ordered arbitrarily, where $y_{2,j}=\frac{x_{1,k}+x_{2,j}}2$ and
$\lambda_{2,j}$, $x_{2,j}$, and $A^2_k$ are as in the definition
of the $\ep$-bush (observe that \eqref{E:Norm1} implies that
$||y_{2,j}||=1$). We label the obtained broken lines by
$\overline{(0)}$ and $\overline{(1)}$, respectively.
\medskip

To get the sequence representing the broken line labelled by
$(0,0)$ we do the following operation with the preliminary
sequence labelled $\overline{(0)}$.

\begin{itemize}

\item Replace each multiple $\lambda y_{1,j}$ present in the
sequence by the two-element sequence
\begin{equation}\label{E:Casey1}\lambda\frac{x_{0,1}}2,\lambda\frac{x_{1,j}}2.\end{equation}

\item Replace each multiple $\lambda y_{2,j}$, with $j\in A^2_k$,
present in the sequence by the two-element sequence
\begin{equation}\label{E:Casey2}\lambda\frac{x_{1,k}}2,\lambda\frac{x_{2,j}}2.\end{equation}
\end{itemize}

To get the sequence representing the broken line labelled by
$(0,1)$ we do the same but changing the order of terms in
\eqref{E:Casey1} and \eqref{E:Casey2}. To get the sequences
representing the broken lines labelled by $(1,0)$ and $(1,1)$, we
apply the same procedure to the broken line labelled
$\overline{(1)}$.
\medskip

We proceed in an obvious way. Suppose that we have already
constructed all broken lines labelled by sequences of length at
most $p$, and all of these broken lines are represented by
sequences consisting of multiples of $x_{n,j}$ with $n\le p$. To
get broken lines labelled by $(a_1,\dots,a_p,0)$ and
$(a_1,\dots,a_p,1)$, we form an intermediate broken line labelled
by $\displaystyle{\overline{(a_1,\dots,a_p)}}$. It is formed as
follows: Each of the vectors $x_{\ell,k}$, where $\ell\le p$ is
replaced by the sequence
\[\{\lambda_{\ell+1,j}y_{\ell+1,j}\}_{j\in A^{\ell+1}_k},\]
where $y_{\ell+1,j}=\frac12(x_{\ell,k}+x_{\ell+1,j})$. The
sequence is multiplied by the same scalar by which $x_{\ell,k}$ is
multiplied in the sequence labelled by $(a_1,\dots,a_p)$.
\medskip

To form the sequence labelled by $(a_1,\dots,a_p,0)$ we replace
each multiple of $y_{\ell+1,j}$ by the corresponding multiple of
the two-element sequence $\frac12x_{\ell,k}$,
$\frac12x_{\ell+1,j}$. To form the sequence labelled by
$(a_1,\dots,a_p,1)$ we replace each multiple of $y_{\ell+1,j}$ by
the corresponding multiple of the two-element sequence
$\frac12x_{\ell+1,j}$, $\frac12x_{\ell,k}$.

\begin{observation}\label{O:Deviat} A broken line labelled by a
sequence $(a_1,\dots,a_p,a_{p+1})$ passes thro\-ugh the vertices
of the broken line labelled by $\overline{(a_1,\dots,a_p)}$. There
is a new vertex $u$ of the broken line labelled by
$(a_1,\dots,a_p,0)$ between each pair of consecutive vertices of
the broken line labelled by $\overline{(a_1,\dots,a_p)}$. There is
a new vertex $v$ of the broken line labelled by
$(a_1,\dots,a_p,1)$ between each pair of consecutive vertices of
the broken line labelled by $\overline{(a_1,\dots,a_p)}$. The
inequality \eqref{E:SumDeviat} generalizes in the following way.
If we add the distances $||u-v||$, where $u$ and $v$ are as above,
over all pairs of consecutive vertices of the broken labelled by
$\overline{(a_1,\dots,a_p)}$, we get at least $\frac\ep2$.
Furthermore, if we add the distances $||u-v||$, where $u$ and $v$
are as above, over some set of pairs of consecutive vertices of
the broken labelled by $\overline{(a_1,\dots,a_p)}$, we get at
least $\frac\ep2\cdot($sum of distances between the consecutive
pairs in the selection$)$.
\end{observation}

The thick family $T_X$ of geodesics whose existence is claimed in
Theorem \ref{T:CustThick} is constructed in two steps.
\medskip

\noindent{\bf Step 1.} Consider all infinite sequences consisting
of $0$ and $1$. For each such sequence we consider a {\it branch}
in the infinite tree $B$ consisting of vertices corresponding to
the finite initial segments of the sequence. The geodesic
corresponding to the branch is the limit (in the sense described
below) of the sequence of broken lines corresponding to vertices
of the branch. To define the limit we observe that all of the
geodesics of the sequence are distance preserving, and thus
$1$-Lipschitz, maps and that for any finite initial segment $I$
all {\it further} geodesics (that is, corresponding to segments
containing $I$ at their beginning) pass through all vertices of
the geodesic corresponding to $I$. In addition, the distance
between the two consecutive vertices of the geodesic corresponding
to a segment $I$ (that is a finite sequences of $0$ and $1$)
containing $t$ terms does not exceed $(\lambda_{\max})^t$ where
$\lambda_{\max}=\max\{\lambda_{n,j}: n, j\}$. On the other hand,
using Definition \ref{D:Bush}, equality \eqref{E:Norm1}, and some
easy inequalities we get that $1-\lambda_{\max}\ge \frac{\ep}{2}$
and thus $\lambda_{\max}\le1-\frac{\ep}{2}$. We conclude that the
sequence of $1$-Lipschitz maps which we consider converges on a
dense subset of $[0,1]$. Hence it converges everywhere on $[0,1]$
to a limit (this is the limit which we meant at the beginning of
this paragraph). Let $T_0$ be the family of all such limits
corresponding to all of the branches of $B$.
\medskip

\noindent{\bf Step 2.} For each finite collection of vertices of
the broken lines corresponding to finite sequences
$(a_1,\dots,a_p)$ consider all geodesics obtained by pasting
together pieces corresponding to geodesics of $T_0$ which join the
corresponding vertices (in the right order, so that the result of
this pasting is again a geodesic joining $0$ and $x_{0,1}$).
Denote the resulting family of geodesics by $T_X$.\medskip

It remains to show that $T_X$ is a thick family of geodesics. Let
$g$ be a geodesic in $T_X$ and let $r_1,\dots,r_n$ be a set of
points on it, let $r_i=g(t_i)$. Let
$[0,1]=\bigcup_{d=1}^w[h_{d-1},h_d]$ be a partition of $[0,1]$ for
which $0=h_0<h_1,\dots<h_w=1$. and the on each of the intervals
$[h_{d-1},h_d]$ the geodesic $g$ coincides with one of the
geodesics in $T_0$.
\medskip

For each of the intervals $[h_{d-1},h_d]$ and the corresponding
geodesic $\widehat g$ in $T_0$ we do the following. Let
$(b_i)_{i=1}^\infty$ be the sequence of $0$ and $1$ corresponding
to the geodesic $\widehat g$. We pick sufficiently large initial
segment $(b_i)_{i=1}^L$ of this sequence such that

\begin{enumerate}

\item $\widehat g(h_{d-1})$ and $\widehat g(h_d)$ are among
vertices of the broken line $g^L$ corresponding to
$(b_i)_{i=1}^L$.

\item All of the $t_i$ which are in $[h_{d-1}, h_d]$ can be
covered by subintervals $[\eta_{i-1},\eta_i]$ of $[h_{d-1},h_d]$
with total length $\le \frac12|h_d-h_{d-1}|$ and $\widehat
g(\eta_{i-1})$, $\widehat g(\eta_i)$ being vertices of the broken
line $g^L$.

\end{enumerate}

The desired geodesic $\widetilde g$ will be picked as follows: on
all intervals $[\eta_{i-1},\eta_i]$ it will coincide with $g$. On
all of the complementary intervals inside $[h_{d-1},h_d]$ we let
$\widetilde g$ to coincide with the geodesic corresponding to the
branch of $B$ which starts with $(b_i)_{i=1}^L$, but for which the
next term (in the infinite sequence) is different from $b_{L+1}$.
Together with Observation \ref{O:Deviat} (see the last sentence of
it), this implies that that the sum of the corresponding
deviations is at least $\frac\ep2\cdot\frac{h_d-h_{d-1}}2$. Doing
the same for all intervals $[h_{d-1},h_d]$, we get that the total
deviation is at least $\frac\ep4$. This completes the proof.
\end{proof}

\begin{remark} Our proof of Theorem \ref{T:CustThick} shows  any
Banach space without the \RNP\ has an equivalent norm in which it
has a thick family of geodesics.
\end{remark}

\section{Acknowledgements}

I would like to thank Beata Randrianantoanina for asking the
question answered by Theorem \ref{T:CustThick}. The research of
the author was supported by NSF DMS-1201269.

\end{large}

\begin{tiny}

\renewcommand{\refname}{
\end{tiny}


\section{References}}

\begin{thebibliography}{BNS12+}

\bibitem[BL00]{BL00} Y. Benyamini, J. Lindenstrauss, {\it Geometric Nonlinear Functional Analysis}, volume 1,
Providence, R.I., AMS, 2000.

\bibitem[Bou79]{Bou79} J.~Bourgain, {\it La propri\'et\'e de Radon-Nikodym}, Publications math\'ematiques
de l'Universit\'e Pierre et Marie Curie No.~{\bf 36}, Paris, 1979.

\bibitem[Bou83]{Bou83} R.\,D.~Bourgin, {\it Geometric aspects of convex sets with the Radon-Nikod\'ym
property}, Lecture Notes in Mathematics, {\bf 993},
Springer-Verlag, Berlin, 1983.

\bibitem[CK06]{CK06} J.~Cheeger, B.~Kleiner,
On the differentiability of Lipschitz maps from metric measure
spaces to Banach spaces, in: {\it Inspired by S.\,S.~Chern},
129--152, Nankai Tracts Math., {\bf 11}, World Sci. Publ.,
Hackensack, NJ, 2006.

\bibitem[CK09]{CK09} J.~Cheeger, B.~Kleiner, Differentiability of Lipschitz
maps from metric measure spaces to Banach spaces with the
Radon-Nikod\'ym property, {\it  Geom. Funct. Anal.}, {\bf  19}
(2009), no. 4, 1017--1028; {\tt arXiv:0808.3249}.

\bibitem[DM13]{DM13} R.~Deville, \'O.~Madiedo, A characterization of the
Radon–Nikod\'ym property, {\it J. Math. Anal. Appl.}, {\bf 405}
(2013), no. 1, 252--258.

\bibitem[DM07]{DM07} R.~Deville, \'E.~Matheron, Infinite games,
Banach space geometry and the eikonal equation, {\it Proc. Lond.
Math. Soc.} (3), {\bf 95} (2007), no. 1, 49--68.

\bibitem[DU77]{DU77} J.~Diestel,  J.\,J.~Uhl, Jr.  {\it Vector measures},
With a foreword by B. J. Pettis. Mathematical Surveys, No. 15.
American Mathematical Society, Providence, R.I., 1977.

\bibitem[LN06]{LN06} J.\,R.~Lee, A.~Naor, $L_p$ metrics on the Heisenberg group and the
Goemans-Linial conjecture, in: {\it Proceedings of the 47th Annual
IEEE Symposium on Foundations of Computer Science}, IEEE,
Piscataway, NJ, 2006, pp. 99--108.

\bibitem[Ost11]{Ost11} M.\,I.~Ostrovskii, On metric characterizations of some classes of Banach
spaces, {\it  C. R. Acad. Bulgare Sci.}, {\bf 64}, (2011), no. 6,
775--784.

\bibitem[Ost13]{Ost13} M.\,I.~Ostrovskii, {\it Metric embeddings: bilipschitz and coarse embeddings into Banach spaces},
de Gruyter Studies in Mathematics, {\bf 49}. Walter de Gruyter \&\
Co., Berlin, 2013.

\bibitem[Ost13+]{Ost13+}  M.\,I.~Ostrovskii, On metric characterizations of the Radon-Nikod\'ym and related properties of Banach
spaces, {\tt arXiv:1302.5968}, 2013.

\bibitem[Pis11]{Pis11} G.~Pisier, {\it Martingales in Banach spaces
(in connection with type and cotype)}, Lecture notes of a course
given at l'Institut Henri Poincar\'e, February 2--8, 2011, 242 pp;
see the web site: {\tt
http://perso-math.univ-mlv.fr/users/banach/Winterschool2011/}

\bibitem[Tex09]{Tex09} {\it Texas 2009 nonlinear problems}, open
problems suggested at the seminar ``Nonlinear geometry of Banach
spaces'', Texas A \&\ M University, August 2009. Available at:
{\tt
http://facpub.stjohns.edu/ostrovsm/Texas2009nonlinearproblems.pdf}

\end{thebibliography}
\end{document}